\documentclass{amsart}
\usepackage{amssymb,amsmath}

\author{Olena Karlova}

\title{Extension of continuous functions to Baire-one functions}

\newtheorem{theorem}{Theorem}[section]
\newtheorem{lemma}[theorem]{Lemma}

\newtheorem{proposition}[theorem]{Proposition}

\newtheorem{question}[theorem]{Question}
\newtheorem{example}[theorem]{Example}

\begin{document}

\maketitle

\begin{abstract}
We introduce the notion of $B_1$-retract and investigate the connection between $B_1$- and $H_1$-retracts.
\end{abstract}

\section{Introduction.}

Recall that a function $f:X\to Y$ between topological spaces $X$
and $Y$  {\it belongs to the first Baire class} or {\it is a Baire-one function}, if it is a pointwise limit of a sequence of continuous functions. A function
$f$  {\it belongs to the first Lebesgue class} or {\it is a Lebesgue-one function} if $f^{-1}(F)$ is
a $G_\delta$-set in $X$ for every closed set  $F$ in $Y$. We shall denote by
$B_1(X,Y)$ ($H_1(X,Y)$) the collection of all functions of the first Baire (Lebesgue) class from $X$ to $Y$.

K.~Kuratowski \cite[p.~445]{Ku1} proved that every continuous function  {$f:E\to \mathbb R$} on an arbitrary subset $E$
of a metric space  $X$ can be extended to a Lebesgue-one function on the whole space. According to Lebesgue-Hausdorff Theorem
\cite[p.~402]{Ku1} the extension is also a Baire-one function.

O.~Kalenda and J.~Spurn\'{y} \cite{Ka} showed that if  $E$ is a Lindel\"{o}f hereditarily Baire subspace or $E$ is a Lindel\"{o}f
$G_\delta$-subspace  of a completely regular space $X$ then every Baire-one function
$f:E\to\mathbb R$ can be extended to a Baire-one function on the whole space.

It was proved in \cite{ZKS} that any Baire-one function with values in a $\sigma$-metrizable space with some additional conditions and defined on a Lindel\"{o}f $G_\delta$-subspace of a normal space
 can be extended to a Baire-one function on the whole space.

The question about the extension of the class of range spaces of extendable functions naturally arises.
In \cite{K} the notion of
$H_1$-retract was introduced. A subset $E$ of a topological space
$X$ is called {\it an $H_1$-retract} if there exists a Lebesgue-one function $r:X\to E$ such that $r(x)=x$ for all $x\in E$. It was shown in \cite{K} that $E$ is an $H_1$-retract of $X$ iff for any topological space  $Y$ and for any continuous function $f:E\to Y$ there exists an extension $g\in
H_1(X,Y)$ of $f$. The following result was established in \cite{K}.

\begin{theorem}\label{Aust} \cite[Corollary 3.3]{K}
A set $E$ is an $H_1$-retract of a completely metrizable space $X$ if and only if $E$ is a $G_\delta$-set in $X$.
\end{theorem}

In this paper we introduce the notion of $B_1$-retract and prove several of its properties. Further, using Theorem \ref{Aust} and the generalization of Lebesgue-Hausdorff Theorem, we find out that $B_1$-retracts and $H_1$-retracts are tightly connected in many cases. And in the last section we give two examples which show that even for subsets of the plane ${\mathbb R}^2$ the notions of retract,
$H_1$-retract and $B_1$-retract are different.

\section{$B_1$-retracts and their properties.}

Recall \cite{Bors} that a subset $E$ of a topological space $X$ is said to be {\it a retract of $X$}
if there exists a continuous function $r:X\to E$ such that $r(x)=x$ for all $x\in E$. The function $r$ is called {\it a retraction} of
$X$ onto $E$. It is well-known that a set $E\subseteq X$ is a
retract of $X$ if and only if for any topological space $Y$ every
continuous function $f:E\to Y$ can be extended to a continuous
function $g:X\to Y$.

We call a subset $E$ of a topological space $X$ {\it a
$B_1$-retract} of $X$ if there exists a Baire-one function
\mbox{$r:X\to E$} such that $r(x)=x$ for all
$x\in E$. We call the function $r$  {\it a
$B_1$-retraction} of $X$ onto $E$.

Note that a composition of a continuous function and a Baire-one function is a Baire-one function. This fact and the definition of a $B_1$-retract  immediately imply the following proposition.

\begin{proposition}

Let $X$ be a topological space. A set $E\subseteq X$ is a
$B_1$-retract of $X$ if and only if for any topological space $Y$ every continuous function $f:E\to Y$ can be extended to a Baire-one function $g:X\to Y$.
\end{proposition}

A subset $A$ of a topological space $X$ is called {\it a regular $G_\delta$-set} \cite{Z} if there exists a sequence $(G_n)_{n=1}^\infty$ of open sets in $X$ such that
$$
A=\bigcap\limits_{n=1}^\infty
G_n=\bigcap\limits_{n=1}^\infty  \overline{G}_n.
$$

We say that a topological space $X$ has {\it a regular
$G_\delta$-diagonal} if its diagonal
$\Delta=\{(x,x):x\in X\}$ is a regular $G_\delta$-set in $X\times X$.

Obviously, every regular $G_\delta$-diagonal is closed
$G_\delta$-set in $X\times X$. Note that every space with a $G_\delta$-diagonal is Hausdorff.

\begin{proposition}

Let $X$ be a topological space with a regular $G_\delta$-diagonal and $E$ be a $B_1$-retract of $X$. Then $E$ is a $G_\delta$-set in $X$.

\end{proposition}

{\bf Proof.} Since the diagonal $\Delta$ is regular $G_\delta$-set, it can be represented as
$\Delta=\bigcap\limits_{n=1}^\infty
G_n=\bigcap\limits_{n=1}^\infty  \overline{G}_n$, where
$(G_n)_{n=1}^\infty$ is a decreasing sequence of open sets in
$X\times X$.
Let $r:X\to E$ be a $B_1$-retraction of $X$ onto $E$.
Consider a function $h:X\to X\times X$, $h(x)=(r(x),x)$.
Then $h\in B_1(X,X\times X)$ and
$E=h^{-1}(\Delta)$. Let $(h_n)_{n=1}^\infty$ be a sequence of continuous functions  such that $h_n(x)\to
h(x)$ for every $x\in X$. We claim that
$$
h^{-1}(\Delta)=\bigcap\limits_{m=1}^\infty
\bigcup\limits_{n=m}^\infty h_n^{-1}(G_m).
$$
Indeed, let $x\in h^{-1}(\Delta)$ and $m\in\mathbb N$. Then
$h(x)\in G_m$. Since $h_n(x)\to h(x)$, there exists a number $n\ge m$ such that
$h_n(x)\in G_m$.
Now let $x$ belongs to the right side of the equality. Then
there exists a sequence $(n_m)_{m=1}^\infty$ such that $n_m\ge
m$ and $h_{n_m}(x)\in G_m$ for every $m\in\mathbb N$. Assume that $h(x)\not\in\Delta$. Then there exists a number $m_0$ such that
$h(x)\not\in \overline{G}_{m_0}$. Since $h_n(x)\to h(x)$, there exists a number $n_0\ge m_0$ such that $h_n(x)\not\in
\overline{G}_{m_0}$ for all $n\ge n_0$. In particular,
$h_{n_{n_0}}(x)\not\in G_{m_0}$ since $n_{n_0}\ge n_0$. Taking into account that
 $G_{n_0}\subseteq G_{m_0}$, we conclude
$h_{n_{n_0}}(x)\not\in G_{n_0}$, a contradiction. Hence, $h(x)\in \Delta$.

Since $G_m$ is an open set in $X\times X$ for every $m$ and $h_n$ is continuous for every $n$, the set $E=h^{-1}(\Delta)$ is
$G_\delta$ in $X$.\hfill$\Box$

Notice that a $B_1$-retract, in general, is not a $G_\delta$-set. For example, let $X$ be a space of all functions $x:[0,1]\to [0,1]$ equipped with a topology of pointwise convergence, $x$ be an arbitrary point from $X$ and $E=\{x\}$. Since $X$ is compact and non-metrizable, the diagonal of $X$ is not a $G_\delta$-set \cite[p.~264]{Eng}, consequently, it is not a regular $G_\delta$-set, and
$E$ is not a $G_\delta$-set in $X$. But, clearly, $E$ is a retract (and, therefore, $E$ is a $B_1$-retract) of the space $X$.

In connection with the previous remark the following open question naturally arises.

\begin{question}
Do there exist a Hausdorff space $X$ with a $G_\delta$-diagonal, but without a regular $G_\delta$-diagonal, and a
$B_1$-retract of $X$, which is not a $G_\delta$-set?

\end{question}

It is well-known that every retract of a connected topological space is also a connected space. The following result states that the same is true for $B_1$-retracts.

\begin{proposition}\label{prop2.4}

Let $X$ be a connected topological space and $E$ be a $B_1$-retract of~$X$. Then $E$ is a connected space.

\end{proposition}

{\bf Proof.} Let $r:X\to E$ be a $B_1$-retraction of $X$ onto $E$ and
$(r_n)_{n=1}^\infty$ be a sequence of continuous functions
$r_n:X\to E$ such that $r_n(x)\to r(x)$ for every $x\in X$.

Assume the contrary. Then $E=E_1\sqcup E_2$, where
$E_1$ and $E_2$ are open in $E$ non-empty sets. Since $r_n$ is a continuous function and $X$ is a connected space, the set $B_n=r_n(X)$ is connected for every $n$. Then $B_n\subseteq E_1$ or $B_n\subseteq E_2$ for every $n\in\mathbb
N$.
Choose any $x\in E_1$. Then $r_n(x)\to r(x)=x$. Since $E_1$ is an open set in $E$, there exists a number
$n_1$ such that $r_n(x)\in E_1$ for all $n\ge n_1$. Then
$r_n(x)\in B_n\cap E_1$, that is, $B_n\subseteq E_1$ for all $n\ge
n_1$.  Analogously, it can be shown that there exists a number $n_2\in\mathbb
N$ such that
$B_n\subseteq E_2$ for all $n\ge n_2$. Hence, $B_n\subseteq
E_1\cap E_2$ for all $n\ge \max\{n_1,n_2\}$, a contradiction.
\hfill$\Box$

A topological space $Y$ is called {\it an equiconnected space}  \cite{Dug} if there exists a continuous function
$\gamma:Y\times Y\times [0,1]\to Y$, which for every $y',y''\in Y$ and
$t\in [0,1]$ satisfies the following properties:
\medskip

$
(i) \,\,\,\,\,\gamma(y',y'',0)=y',
$
\medskip

$
(ii)\,\,\,\, \gamma(y',y'',1)=y'',
$
\medskip

$
(iii)\,\,\, \gamma(y',y',t)=y'.
$

\medskip
We need the following auxiliary fact from \cite{ZKS}.

\begin{lemma}\label{l1} \cite[Lemma 2.1]{ZKS} Let $X$ be a normal space, $Y$ be an equiconnected space, $(F_i)_{i=1}^n$ be disjoint closed sets in $X$ and $g_i:X\to Y$ be a continuous function for every $1\le i\le n$.
Then there exists a continuous function  \mbox{$g:X\to Y$} such that
$g(x)=g_i(x)$ on $F_i$ for every $1\le i\le n$.
\end{lemma}

\begin{proof}
  The proof is by induction on $n$. Let $n=2$. Since $F_1$ and $F_2$ are disjoint and closed, by Urysohn's Lemma
\cite[p.~41]{Eng} there exists a continuous function $\varphi:X\to
[0,1]$ such that $\varphi(x)=0$ on $F_1$ and
$\varphi(x)=1$ on $F_2$. The space $Y$ is equiconnected, therefore
there exists a continuous function $\gamma:Y\times Y\times
[0,1]\to Y$, which satisfies  { (i) -- (iii)}. Let
$$
g(x)=\gamma(g_1(x),g_2(x),\varphi(x))
$$
for every
$x\in X$.
Clearly, $g:X\to Y$ is continuous. If $x\in F_1$, then
$\varphi(x)=0$ and $g(x)=g_1(x)$. If $x\in F_2$, then
$\varphi(x)=1$ and $g(x)=g_2(x)$.

Assume the lemma is true for all  $1\le k<n$. We will prove it for
$k=n$. According to the assumption, there exists a continuous function $\tilde{g}:X\to Y$ such that
$\tilde{g}|_{F_i}=g_i$ for every $1\le i<n$. Since
$F=\bigcup\limits_{i=1}^{n-1}F_i$ and $F_n$ are disjoint and closed in
$X$, by the assumption there exists such a continuous function
$g:X\to Y$ that $g|_F=\tilde{g}$ and
$g|_{F_n}=g_n$. Then $g|_{F_i}=g_i$ for every $1\le i\le
n$.
\end{proof}

We call a subset $A$ of a topological space $X$ {\it an ambiguous set} if it is simultaneously
$F_\sigma$ and $G_\delta$ in $X$. Recall that a topological space $X$ is called {\it perfectly normal} if it is normal and every closed subset of $X$ is a $G_\delta$-set.

\begin{theorem} Let $X$ be a perfectly normal space, $E\subseteq
X$ be an equiconnected space, $E=\bigcup\limits_{n=1}^\infty
E_n$, and the following conditions hold:

(1) $E_n\cap E_m=\O$ if $n\ne m$;

(2) $E_n$ is an ambiguous set in $E$ for every $n\in\mathbb
N$;

(3) $E_n$ is a $B_1$-retract of $X$ for every $n\in\mathbb
N$;

(4) $E$ is a $G_\delta$-set in $X$.

Then  $E$ is a $B_1$-retract of $X$.
\end{theorem}

{\bf Proof.} From the condition (2) and \cite[p.~359]{Ku1} it follows that for every
$n\in\mathbb N$ there exists an ambiguous set $C_n$ in $X$ such that $C_n\cap E=E_n$. Let $D_1=C_1$ and
$D_n=C_n\setminus \bigcup\limits_{k<n}C_k$ if $n\ge 2$. Then
$D_n$ is an ambiguous set for every $n$. Moreover,
$D_n$ are disjoint sets and $D_n\cap E=E_n$ for every $n\in\mathbb N$.
Since $X\setminus E$ is an $F_\sigma$-set in $X$, there exists a sequence $(F_n)_{n=1}^\infty$ of closed subsets of $X$ such that $X\setminus E=\bigcup\limits_{n=1}^\infty F_n$. Let
$X_1=F_1\cup D_1$, and  $X_n=(F_n\cup D_n)\setminus
(\bigcup\limits_{k<n}(F_k\cup D_k))$ if $n\ge 2$.
Obviously, $X_n$ is an ambiguous set in $X$ for every $n$,
$X_n\cap X_m=\O$ if $n\ne m$, and $X=\bigcup\limits_{n=1}^\infty
X_n$.

We will show that $X_n\cap E=E_n$ for every $n\in\mathbb N$. Indeed,
if $x\in X_n\cap E$, then
$$
x\in (F_n\cup D_n)\cap E=(F_n\cap E)\cup (D_n\cap E)=D_n\cap
E=E_n.
$$
If $x\in E_n$, then $x\in D_n\cap E$, therefore $x\in D_n$ and
$x\not\in F_m$ for all $m$. Moreover, $x\not\in D_k$ for all
$k<n$, since $D_n\cap D_k=\O$ if $n\ne k$. Hence, $x\in
X_n\cap E$.

According to (3), there exists a sequence of
$B_1$-retractions $r_n:X\to E_n$. Let $r(x)=r_n(x)$
if $x\in X_n$. We will show that $r\in B_1(X,E)$.

For every $n\in\mathbb N$ there exists a sequence
$(r_{n,m})_{m=1}^\infty$ of continuous functions $r_{n,m}:X\to
E_n$ such that $\lim\limits_{m\to\infty} r_{n,m}(x)=r_n(x)$ for every $x\in X$. Notice that $\lim\limits_{m\to\infty}
r_{n,m}(x)=r(x)$ on $X_n$. The set $X_n$ is $F_\sigma$, therefore,
for every $n$ there exists an increasing sequence
$(B_{n,m})_{m=1}^\infty$ of closed subsets $B_{n,m}$ of $X$ such that $X_n=\bigcup\limits_{m=1}^\infty B_{n,m}$. Let
$A_{n,m}=\O$ if $n>m$, and $A_{n,m}=B_{n,m}$ if $n\le m$. Then Lemma \ref{l1} implies that for every $m\in\mathbb N$ there exists a continuous function  $g_m:X\to E$ such that
$g_m|_{A_{n,m}}=r_{n,m}$ since a family $\{A_{n,m}:n\in
{\mathbb N}\}$ is finite for every $m\in {\mathbb N}$.

It remains to prove that $g_m(x)\to r(x)$ on $X$. Fix $x\in X$.
Then $x\in X_n$ for some $n\in\mathbb N$. The sequence $(A_{n,m})_{m=1}^\infty$ is increasing, and, in consequence,
there exists $m_0$ such that $x\in A_{n,m}$ for every $m\ge m_0$.
Then $g_m(x)=r_{n,m}(x)$ for all $m\ge m_0$. Hence,
$\lim\limits_{m\to\infty}
g_m(x)=\lim\limits_{m\to\infty}r_{n,m}(x)=r(x)$. Therefore,
$r\in B_1(X,Y)$.

It is easy to see that $r(x)=x$ for all $x\in E$. Hence,
$r$ is a $B_1$-retraction of $X$ onto $E$.\hfill$\Box$

\section{The connection between $B_1$-retracts and $H_1$-retracts.}

Recall that a family ${\mathcal A}$ of subsets of a topological space $X$ is {\it discrete} if every point
$x\in X$ has a neighbourhood $U$ that intersects at most one of the sets $A\in {\mathcal A}$. A family ${\mathcal A}=(A_i:i\in
I)$ of subsets of a topological space $X$ is said to be {\it strongly discrete} \cite{Ve} if there is a discrete family ${\mathcal
G}=(G_i:i\in I)$ of open sets in $X$ such that
$\overline{A}_i\subseteq G_i$ for any $i\in I$. A family
${\mathcal A}$ is {\it $\sigma$-discrete (strongly
$\sigma$-discrete)} if it can be represented as the union of countably many
discrete (strongly discrete) families in $X$.

A family ${\mathcal B}$ of subsets of topological space $X$ is
{\it a base for a function} \mbox{$f:X\to Y$} if for any open set $V$ in $Y$ there exists a subfamily
\mbox{${\mathcal B}_V\subseteq {\mathcal B}$} such that
$f^{-1}(V)=\bigcup {\mathcal B}_V$. If
${\mathcal B}$ is (strongly) $\sigma$-discrete then it is called
{\it (strongly) $\sigma$-discrete base for $f$} and function
\mbox{$f:X\to Y$} with (strongly) $\sigma$-discrete base is called
{\it (strongly) $\sigma$-discrete function}. We shall denote by $\Sigma(X,Y)$ ($\Sigma^*(X,Y)$) the set of all (strongly)
$\sigma$-discrete functions from $X$ to $Y$.

A topological space $X$ is {\it collectionwise normal} if $X$ is $T_1$-space and for each discrete family
$(F_i:i\in I)$ of closed sets there exists a discrete family
$(G_i:i\in I)$ of open sets such that $F_i\subseteq G_i$
for every $i\in I$. It is easy to see that a space is collectionwise normal if and only if every discrete family of its subsets is strongly discrete.

Note that any function with values in a second countable topological space is strongly $\sigma$-discrete. R.~Hansell
\cite{Ha3} proved that every Lebesgue-one function with a complete metric domain space and a metric range space is $\sigma$-discrete. Taking into account that a complete metric space is collectionwise normal, we obtain the strongly $\sigma$-discreteness of such a function.

Recall that a topological space $X$ is {\it arcwise connected} if for any two points $x$ and $y$ from $X$
there exists a continuous function $f:[0,1]\to X$ such that $f(0)=x$ and
$f(1)=y$. A space $X$ is called {\it locally arcwise connected} if for every $x\in X$ and for any its neighbourhood $U$ there exists a neighbourhood $V$ of $x$ such that for each $y\in V$ there is a continuous function  $f:[0,1]\to U$ such that $f(0)=x$ and $f(1)=y$.

We shall need the following results of L.~Vesel\'{y} \cite{Ve} and M.~Fosgerau~\cite{F} concerning the equality between Baire and Lebesgue classes.

\begin{theorem}\label{Vesely}\cite[Theorem 3.7(i)]{Ve} Let $X$ be a normal space, $Y$ be an arcwise connected and locally arcwise connected metric space. Then $$H_1(X,Y)\cap \Sigma^*(X,Y)=B_1(X,Y).$$
\end{theorem}

\begin{theorem}\label{Fosgerau}\cite[Theorem 2]{F} Let $Y$ be a complete metric space. Then the following conditions are equivalent:

(i) $Y$ is connected and locally connected;

(ii) $H_1(X,Y)\cap \Sigma(X,Y)=B_1(X,Y)$ for any metric space  $X$.
\end{theorem}

\begin{theorem}
Let $X$ be a normal space and $E$ be an arcwise connected and locally arcwise connected metrizable ambiguous subspace of $X$. If one of the following conditions holds

(i) $E$ is separable, or

(ii) $X$ is collectionwise normal,

\noindent then $E$ is a $B_1$-retract of $X$.

\end{theorem}

{\bf Proof.} Fix any point $x^*\in E$ and define
$$
r(x)=\left\{\begin{array}{ll}
  x, & \mbox{if}\,\, x\in E, \\
  x^*, & \mbox{if}\,\, x\in X\setminus E.\\
\end{array}
\right.
$$
We claim that $r$ is an $H_1$-retraction of $X$ onto $E$. Indeed, take an arbitrary open set $V$ in $E$. If $x^*\not\in V$, then $r^{-1}(V)=V$. Since $E$ is metrizable, $V$ is an $F_\sigma$-set in $E$. Moreover, $E$ is $F_\sigma$ in $X$, therefore, $V$ is $F_\sigma$ in $X$. If $x^*\in V$, then $r^{-1}(V)=V\cup (X\setminus E)$. Since $V$ and $X\setminus E$ are $F_\sigma$-sets in $X$, $r^{-1}(V)$ is also an $F_\sigma$-set in $X$.

{\it (i)} Since $E$ is a second countable space, a function $r$ is strongly $\sigma$-discrete.
According to Theorem \ref{Vesely},  $r\in B_1(X,E)$.

{\it (ii)} Show that $r:X\to E$ is strongly
$\sigma$-discrete.
Since $E$ is $F_\sigma$ in $X$, there exists an increasing sequence $(F_n)_{n=1}^\infty$ of closed subsets of $X$ such that
$E=\bigcup\limits_{n=1}^\infty {F_n}$.  Note that every metrizable space has a $\sigma$-discrete base according to Bing's Theorem \cite[p.~282]{Eng}, therefore, for every $n$ we can
choose a $\sigma$-discrete base ${\mathcal U}_n$ of $F_n$.
Then ${\mathcal U}_n=\bigcup\limits_{m=1}^\infty {\mathcal
U}_{n,m}$, where $({\mathcal U}_{n,m})_{m=1}^\infty$ is a sequence of discrete families in $F_n$, $n\in\mathbb N$.
The set $F_n$ is closed in $X$, and, consequently,
${\mathcal U}_{n,m}$ is discrete in $X$,
$n,m\in\mathbb N$. Since $X$ is collectionwise normal, the family
${\mathcal U}_{n,m}$ is strongly discrete in $X$. Then the families ${\mathcal U}_n$ and ${\mathcal
U}=\bigcup\limits_{n=1}^\infty {\mathcal U}_n$ are
strongly $\sigma$-discrete in $X$.
Let ${\mathcal B}={\mathcal U}\cup\{X\setminus E\}$. Then
${\mathcal B}$ is a strongly $\sigma$-discrete family in $X$.

We prove that ${\mathcal B}$ is a base for $r$. Let
$U$ be an open set in $E$. Then $U=\bigcup\limits_{n=1}^\infty
(U\cap F_n)$. Since $U\cap F_n$ is an open set in $F_n$ for every $n$, there exists a subfamily ${\mathcal U}_{n,U}\subseteq
{\mathcal U}_n$ such that $U\cap F_n=\bigcup\limits {\mathcal
U}_{n,U}$. If $x^*\not\in U$, then
$r^{-1}(U)=U=\bigcup\limits_{n=1}^\infty \bigcup\limits {\mathcal
U}_{n,U}$. If $x^*\in U$, then $r^{-1}(U)=U\cup (X\setminus
E)=\bigcup\limits_{n=1}^\infty \bigcup\limits {\mathcal
U}_{n,U}\cup (X\setminus E)$. Therefore, ${\mathcal B}$ is a strongly
$\sigma$-discrete base for $r$. Hence, $r\in
\Sigma^*(X,E)$.

By Theorem \ref{Vesely}, $r\in B_1(X,E)$.\hfill$\Box$

\begin{theorem}\label{t2.9}

Let $X$ be a complete metric space and $E\subseteq X$ be an arcwise connected and locally arcwise connected $G_\delta$-set. Then $E$ is a
$B_1$-retract of $X$.

\end{theorem}

{\bf Proof.} Theorem \ref{Aust} implies that there exists an $H_1$-retraction $r:X\to E$
of $X$ onto $E$. Since $X$ is complete, Hansell's Theorem \cite[Theorem 3]{Ha3} implies that  $r$ is strongly $\sigma$-discrete. According to Theorem
\ref{Vesely}, $r\in B_1(X,E)$ and, therefore, $r$ is a $B_1$-retraction of $X$ onto $E$.
\hfill$\Box$

\section{Examples.}

It is well-known that any retract of a locally connected space is also a locally connected space, and any retract of an arcwise connected space is an arcwise connected space too. We give an example which shows that for $B_1$-retracts it is not true.

We first prove the following auxiliary fact.

\begin{lemma}\label{L} Let $A=[a,b]\times [c,d]$, $B_1=A\setminus
\left((a,b)\times(c,d]\right)$ and {$B_2=A\setminus
\left((a,b)\times[c,d)\right)$}. Then $B_i$ is a retract of $A$ for every $i=1,2$.
\end{lemma}

{\bf Proof.} Let $p_1=\left(\frac{a+b}{2},d+1\right)$ and
$p_2=\left(\frac{a+b}{2},c-1\right)$.  For $i=1,2$ and $(x,y)\in A$ denote by $\ell_i(x,y)$ the line which connects  $(x,y)$ and $p_i$.
Consider a function \mbox{$\varphi_i: A\to B_i$} such that $\varphi_i(x,y)$ is the point
of the intersection of $\ell_i(x,y)$ with $B_i$, $i=1,2$.
It is easy to see that for every $i=1,2$ the function $\varphi_i$ is a retraction of $A$ onto~$B_i$.
\hfill$\Box$

\begin{example} \label{ex2.5}

There exists a connected closed $B_1$-retract of $[0,1]^2$ which is neither arcwise connected nor locally connected.

\end{example}

{\bf Proof.} Let
$$E=(\{0\}\times[0,1])\cup (\bigcup\limits_{n=1}^\infty(\{\frac 1n\}\times [0,1])\cup
([\frac{1}{2n},\frac{1}{2n-1}]\times\{1\})\cup([\frac{1}{2n+1},\frac{1}{2n}]\times\{0\})).
$$

It is not difficult to check that $E$ is a connected closed subset of $[0,1]^2$, which is neither arcwise connected nor locally connected.
We show that $E$ is a $B_1$-retract of $[0,1]^2$.

Let \\
\centerline{$ A_n=\left[\frac{1}{n+1},\frac{1}{n}\right]\times
[0,1],\quad\quad B_n=\left[0,\frac
{1}{n+1}\right]\times[0,1],\quad n\ge 1, $}

\medskip

\centerline{$E_n=\left(\left[\frac{1}{n+1},\frac{1}{n}\right]\times
[0,1]\right)\setminus
\left(\left(\frac{1}{n+1},\frac{1}{n}\right)\times [0,1)\right)$
if  $n$ is an odd number,}

\medskip

\centerline{$E_n=\left(\left[\frac{1}{n+1},\frac{1}{n}\right]\times
[0,1]\right)\setminus
\left(\left(\frac{1}{n+1},\frac{1}{n}\right)\times (0,1]\right)$
if $n$ is an even number.}

By Lemma \ref{L}, $E_n$ is a retract of
$A_n$ for every $n$. Denote by $\varphi_n$ a retraction
of $A_n$ onto $E_n$, $n\in\mathbb N$. Let $\psi_n$ be a continuous function $\psi_n:B_n\to \{\frac{1}{n+1}\}\times [0,1]$, $\psi_n(x,y)=(\frac{1}{n+1},y)$.

For every $n\ge 1$ and $x,y\in [0,1]$ define
$$
r_n(x,y)=\left\{\begin{array}{ll}
  \varphi_k(x,y), & \,\,(x,y)\in A_k, 1\le k\le n,\\
  \psi_n(x,y), &  \,\, (x,y)\in B_n.\\
\end{array}
\right.
$$
The function $r_n:[0,1]^2\to \bigcup\limits_{k\le n} E_k$
is correctly defined and continuous for every $n$, since
$\varphi_k|_{A_k\cap A_{k+1}}=\varphi_{k+1}|_{A_k\cap A_{k+1}}$,
 $1\le k<n$, and $\varphi_n|_{A_n\cap B_n}=\psi_n|_{A_n\cap
B_n}$.

We show that $(r_n)_{n=1}^\infty$ is a pointwise convergent sequence on
$[0,1]^2$. Fix an arbitrary $(x,y)\in [0,1]^2$. If $x\ne 0$,
then there exists  $n_0$ such that $(x,y)\in A_{n_0}$. Then
$r_n(x,y)=\varphi_{n_0}(x,y)$ for all $n\ge n_0$, that is
$r_n(x,y)\mathop{\to}\limits_{n\to\infty} \varphi_{n_0}(x,y)$.
Note that if $(x,y)\in E$, then
$\varphi_{n_0}(x,y)=(x,y)$. If $x=0$, then
$r_n(x,y)=\psi_n(x,y)=(\frac{1}{n+1},y)\mathop{\to}\limits_{n\to\infty}
(0,y)=(x,y)$. Hence, there exists $\lim\limits_{n\to\infty} r_n(x,y)$ for each $(x,y)\in [0,1]^2$. We remark that
$\lim\limits_{n\to\infty} r_n(x,y)=(x,y)$ on~$E$. Moreover, since $E$ is closed, $\lim\limits_{n\to\infty} r_n(x,y)\in E$ for all $(x,y)\in [0,1]^2$.

Set
$r(x,y)=\lim\limits_{n\to\infty}r_n(x,y)$ for every $(x,y)\in [0,1]^2$.  Then
$r:[0,1]^2\to E$ is a  $B_1$-retraction of $[0,1]^2$ onto $E$.\hfill$\Box$

Note that for the set $E$ from the previous example there exists an
$H_1$-retraction \mbox{$r:[0,1]^2\to E$}. Though $E$ is a complete metric separable connected space, we cannot apply Theorem \ref{Fosgerau} for $r$ since $E$ is not locally connected. Therefore, it is natural to ask: is every connected $H_1$-retract of a complete metric separable connected and locally connected space its $B_1$-retract? The following example shows that the answer to this question is negative.

\begin{example}  There exists a connected  $H_1$-retract of a complete metric separable connected and locally connected space which is not its $B_1$-retract.
\end{example}

{\bf Proof.} Let ${\mathbb Q}_0={\mathbb Q}\cap [0,1]=\{q_n:n\in {\mathbb N}\}$.
For every $n\in\mathbb N$ consider a function $f_n:[0,1]\to
[-1,1]$, $f_n(x)=\sin \frac {1}{x-q_n}$ if $x\ne q_n$, and
$f_n(q_n)=0$. For every $x\in [0,1]$ define
$$
f(x)=\sum\limits_{n=1}^\infty \frac{1}{2^n}f_n(x),
$$
$$
X=[0,1]\times [-1,1]
\quad\mbox{and}\quad E={\rm Gr}(f)=\{(x,y)\in X: y=f(x)\}.
$$

For every $n$ the function $f_n:[0,1]\to [-1,1]$ belongs to the first Baire class, since it is discontinuous only in one point $x=q_n$
(see for instance \cite[p.~384]{Nat}). Moreover, $f_n$ is a Darboux function (i.e. $f_n(A)$ is connected for every connected set $A\subseteq [0,1]$) \cite[p.~13]{Bru}. Then
$f$ is a Darboux Baire-one function, as the sum of the uniform convergent series of Darboux Baire-one functions \cite[p.~13]{Bru}.
Therefore, the set $E$, as the graph of $f$, is connected  \cite[p.~9]{Bru} and
$G_\delta$ \cite[p.~393]{Ku1} in $X$. Hence,  according to Theorem \ref{Aust} the set $E$ is an $H_1$-retract of $X$.

We prove that $E$ is not a $B_1$-retract of $X$.
Assume that there exists a function $r\in B_1(X,E)$ such that $r(p)=p$
for all $p\in E$. Let $(r_n)_{n=1}^\infty$ be a sequence of continuous functions $r_n:X\to E$, which is pointwise convergent to $r$ on $X$. Since $X$ is compact and connected, $E_n=r_n(X)$ is also compact and connected for every $n$. Note that at least one of $E_n$ contains more than one point. Indeed, assume that all the sets  $E_n$ consist of one point, i.e.
$E_n=\{p_n\}$, where $p_n\in E$, $n\in \mathbb N$.
Choose two different points $p'$ and $p''$ from $E$. Then
$p_n=r_n(p')\mathop{\to}\limits_{n\to\infty} p'$ and
$p_n=r_n(p'')\mathop{\to}\limits_{n\to\infty} p''$, a contradiction. Hence, there exists a number $n_0$ such that
$E_{n_0}$ contains at least two different points (to be more precise, the cardinality of $E_{n_0}$ is equal to $\frak c$ since $E_{n_0}$ is a connected set).

Now fix $p,q\in E_{n_0}$. Since $E_{n_0}\subseteq {\rm Gr}(f)$, the points $p$ and $q$ are represented as $p=(a,f(a))$ and $q=(b,f(b))$. Without loss of generality we can assume that $a<b$. Note that
$(x,f(x))\in E_{n_0}$ for any $x\in (a,b)$. Indeed, if there exists a point $x_0\in (a,b)$ such that $(x_0,f(x_0))\not\in
E_{n_0}$, then the line $x=x_0$ does not intersect $E_{n_0}$. Then, since $E_{n_0}$ is connected, it should be completely contained either in the left hand half-plane, or in the right hand half-plane with respect to the line $x=x_0$. But this contradicts the fact that both $p$ and $q$ belong to
$E_{n_0}$.

Since ${\mathbb Q}_0$ is dense in $[0,1]$, there exists a number $k$ such that
$q_k\in (a,b)$. Note that  $f$ is discontinuous in every point of ${\mathbb Q}_0$, in particular, it is discontinuous in $q_k$.
Then there exists a sequence $(x_n)_{n=1}^\infty$, $x_n\in (a,b)$, such that $\lim\limits_{n\to\infty} x_n=q_k$, but
$\lim\limits_{n\to\infty} f(x_n)\ne f(q_k)$. Since
$(x_n,f(x_n))\in E_{n_0}$ for every $n$  and $E_{n_0}$ is closed,
the point $(\lim\limits_{n\to\infty} x_n,
\lim\limits_{n\to\infty} f(x_n))=(q_k, \lim\limits_{n\to\infty}
f(x_n))$ also belongs to $E_{n_0}$. But then it must be equal to $(q_k, f(q_k))$, a contradiction.

Hence, $E$ is not a $B_1$-retract of $X$.\hfill$\Box$

\end{document}